\documentclass[11pt]{article}

\usepackage{amsmath, amssymb, amsthm, mathtools}
\usepackage{geometry}
\usepackage{mathrsfs}
\usepackage{enumitem}
\usepackage[authoryear]{natbib}
\usepackage{xcolor}
\usepackage[colorlinks=true,citecolor=blue,urlcolor=black]{hyperref}

\title{Stochastic Two-Species Chemotaxis--Competition: Global Well-Posedness and Continuous Dependence}
\author{
	Yiming Jiang$^{\mathrm{a}}$, Haohang Li$^{\mathrm{b},*}$, Yawei Wei$^{\mathrm{a}}$
	\\[0.4em]
	\small\textit{$^{\mathrm{a}}$School of Mathematical Sciences and LPMC, Nankai University, Tianjin, 300071, China}
	\\
	\small\textit{$^{\mathrm{b}}$School of Statistics and Data Science, Nankai University, Tianjin, 300071, China}
	\\[0.3em]
}
\date{}

\newtheorem{definition}{Definition}[section]
\newtheorem{theorem}[definition]{Theorem}
\newtheorem{lemma}[definition]{Lemma}

\newtheorem{assumption}[definition]{Assumption}

\numberwithin{equation}{section}

\begin{document}

	\maketitle
	\begingroup
	\renewcommand{\thefootnote}{\fnsymbol{footnote}}
	\footnotetext[1]{Corresponding author. E-mail: \texttt{1120240077@mail.nankai.edu.cn}}
	\endgroup

\begin{abstract}
	In this paper, we investigate a stochastic two species competition Keller--Segel model. Chemotactic movement, interspecific competition, and environmental fluctuations may act simultaneously when two species respond to the same chemical signal. We consider a stochastic two-species chemotaxis--competition system on a bounded smooth domain $\mathcal O\subset\mathbb R^n$, $n\in\{1,2\}$. Both species contribute to the production of the signal and move along its gradient; their population equations also include nonlinear competition and multiplicative noise. If the self-damping exponents exceed two, we prove that the system has a unique global adapted nonnegative mild solution. We also establish continuous dependence in probability on admissible initial data over finite time intervals. The results prove the global well-posedness under the stated assumptions and we still show that small changes in the initial populations, in probability, to small changes in the solution. Due to the two coupled chemotactic fluxes and the stochastic terms, the $L^p$ estimates contain terms with no fixed sign. They are controlled by parabolic estimates for the signal equation, superquadratic self-damping, and stochastic convolution estimates.
\end{abstract}

	\begingroup
	\renewcommand{\thefootnote}{}
	\footnotetext{\textit{Email addresses:} 
		\texttt{ymjiangnk@nankai.edu.cn} (Yiming Jiang), \texttt{1120240077@mail.nankai.edu.cn} (Haohang Li), \texttt{weiyawei@nankai.edu.cn} (Yawei Wei)}
	\addtocounter{footnote}{-1}
	\endgroup

	\section{Introduction}

	Chemotaxis is the directed movement of cells or organisms along a chemical gradient. Within a community of interacting species, chemotactic movement may occur together with population growth, competition for resources, and chemical production. Experimental observations show that bacterial chemotaxis can produce strongly nonuniform spatial distributions \citep{Adler1966,BudreneBerg1991}. Chemotaxis toward autoinducer-2 has also been linked to aggregation, dual-species biofilm formation, and coexistence of bacterial strains \citep{LaganenkaColinSourjik2016,LaganenkaSourjik2018,LaganenkaEtAl2023}. Random changes in the environment provide a further source of variation, which leads naturally to stochastic population models.

	We consider the stochastic two-species chemotaxis--competition system
	\begin{equation}\label{eq:main}
		\left\{
		\begin{aligned}
			 d u
			={}&
			\Big[
			\Delta u-\chi_1\nabla\cdot(u\nabla w)
			+\mu_1u\bigl(1-u^{\zeta_1}-a_1v^{\eta_1}\bigr)
			\Big]\, d t
			+\sigma_1(u)\, d W_1(t),
			\\
			 d v
			={}&
			\Big[
			\Delta v-\chi_2\nabla\cdot(v\nabla w)
			+\mu_2v\bigl(1-v^{\zeta_2}-a_2u^{\eta_2}\bigr)
			\Big]\, d t
			+\sigma_2(v)\, d W_2(t),
			\\
			\partial_t w
			={}&
			\Delta w-w+\alpha u+\beta v,
		\end{aligned}
		\right.
	\end{equation}
	The system is posed in a bounded smooth domain $\mathcal O\subset\mathbb R^n$, $n\in\{1,2\}$, with homogeneous Neumann boundary conditions and nonnegative initial data. The functions $u=u(x,t)$ and $v=v(x,t)$ are the population densities of the two competing species, and $w=w(x,t)$ is the signal concentration. The constants $\chi_1$ and $\chi_2$ are the chemotactic sensitivities, while $\alpha$ and $\beta$ are the signal production rates. The powers $u^{\zeta_1}$ and $v^{\zeta_2}$ give nonlinear self-damping, and the terms containing $a_1v^{\eta_1}$ and $a_2u^{\eta_2}$ describe interspecific competition. We assume that the noise coefficients are globally Lipschitz as $\ell^2$-valued functions, vanish at zero, and have linear growth. In particular, the noise vanishes when the corresponding population density is zero.

	The deterministic Keller--Segel system has been studied in many forms; see \citep{Horstmann2003KellerSegelI,Horstmann2004KellerSegelII,HillenPainter2009} for surveys and \citep{HORSTMANN200552} for the boundedness and blow-up problem. Logistic damping can prevent excessive concentration. For a one-species system, \citet{TelloWinkler2007} proved global boundedness in low dimensions and under suitable damping conditions. \citet{Winkler2014Stability} later obtained asymptotic stability for a fully parabolic system when logistic damping is sufficiently strong.

	Two-species systems combine chemotactic drift with Lotka--Volterra competition. \citet{TelloWinkler2012} studied stabilization when two populations respond to the same signal, and \citet{StinnerTelloWinkler2014} treated competitive exclusion. For fully parabolic systems, global boundedness and spatial pattern formation were studied in one dimension by \citet{HuWangYangZhang2015} and in two dimensions by \citet{ZhangHuangXia2015}. \citet{BaiWinkler2016} proved global boundedness in dimensions one and two for the fully parabolic system with quadratic Lotka--Volterra damping and also studied equilibration. Further stability and boundedness results for weak competition and parabolic--parabolic--elliptic systems can be found in \citep{BlackLankeitMizukami2016,Mizukami2018}. More recent work also considers different forms of signal production and consumption \citep{ZhangLiu2025}. These results show that the outcome depends on the dimension, the signal equation, and the relative strength of chemotaxis and population damping.

	Stochastic Keller--Segel equations have been considered with several types of noise. \citet{Chavanis2010} derived stochastic chemotaxis models that retain fluctuations absent from the mean-field equation. \citet{HuangQiu2021} obtained a stochastic Keller--Segel equation from an interacting particle system and proved well-posedness together with a mean-field limit. \citet{HausenblasMukherjeeTran2022} constructed martingale solutions for a one-dimensional Keller--Segel system driven by spatial Wiener processes. Related existence and uniqueness results are available for two-dimensional stochastic chemotaxis--Navier--Stokes equations \citep{ZhaiZhang2020}. For a one-dimensional chemorepulsion equation with additive space--time white noise, \citet{ChevyrevHamblyMayorcas2023} proved global well-posedness and studied its invariant measure. Stochastic competitive reaction--diffusion systems without chemotaxis were studied by \citet{NguyenYin2021}.

	The effect of noise depends on its form. \citet{MisiatsStanzhytskyiTopaloglu2022} obtained global weak solutions for small data in a divergence-form setting and finite-time blow-up with positive probability for a different perturbation. Blow-up under conservative noise was also proved by \citet{MayorcasTomasevic2023}. \citet{CHEN2025113531} established well-posedness for a one-species stochastic chemotaxis system with a logistic source or nonlinear noise. The present paper concerns two stochastic population equations coupled through interspecific competition and a fully parabolic signal equation.

	The question raised by system~\eqref{eq:main} is whether the two population densities remain globally defined and nonnegative when signal-driven aggregation, competition, and environmental fluctuations act together. The deterministic two-species papers cited above study competition and a shared signal without stochastic forcing. The cited stochastic chemotaxis papers treat one population or structurally different equations, while stochastic competitive systems without chemotaxis do not contain signal-driven aggregation. In system~\eqref{eq:main}, multiplicative noise acts on both population equations, and the two species interact through signal-driven movement and nonlinear competition. These features distinguish the present model from the cited systems.

	For $n\in\{1,2\}$, we prove that if $\zeta_1,\zeta_2>2$ and the competition exponents are at least one, then system~\eqref{eq:main} has a unique global adapted nonnegative mild solution. The initial population densities are continuous, the initial signal belongs to $W^{1,\infty}(\mathcal O)$, and the initial data are nonnegative with the required finite moments. No smallness condition is imposed on the initial data or on the positive coefficients. We further prove continuous dependence in probability over every finite time interval. Here convergence is measured in $C(\overline{\mathcal O})\times C(\overline{\mathcal O})\times W^{1,q}(\mathcal O)$ for $q>n$, and the family of solutions is assumed to satisfy a uniform  bound. Thus the model does not lose well-posedness in finite time under the stated damping and noise assumptions. Nonnegativity preserves the meaning of $u$, $v$, and $w$ as population densities and signal concentration, while continuous dependence gives stability with respect to the initial state. The precise statements are given in Theorems~\ref{thm:main-result} and~\ref{thm:initial-data-continuity}.

	Due to the two chemotactic fluxes and the stochastic terms, the $L^p$ It\^o identities contain coupled terms with no fixed sign. In addition, the estimates needed for global existence must be uniform in the cut-off level. Parabolic estimates for the signal equation control $\nabla w$ by the population densities, and the assumption $\zeta_i>2$ allows the resulting powers to be absorbed by self-damping. Heat-semigroup and stochastic convolution estimates then yield  $L^\infty$ bounds for $u$ and $v$, which exclude a finite maximal time. Continuous dependence follows by applying local difference estimates until either solution exits a fixed bounded set and then using the uniform moment bound to remove the localization.

	The paper is organized as follows. Section~2 introduces the notation and preliminary estimates. In Section~3, we construct the maximal local solution and prove nonnegativity and uniqueness. Section~4 establishes the uniform a priori estimates and proves Theorem~\ref{thm:main-result}. Section~5 proves the continuous dependence result, and Section~6 concludes the paper.

	\section{Preliminaries}

	Let $\mathcal{O} \subset \mathbb{R}^n, n=1,2$ be a bounded domain with smooth boundary. We denote by $L^p(\mathcal{O})$ and $W^{k,p}(\mathcal{O})$ the standard Lebesgue and Sobolev spaces.

	For $p \in (1,\infty)$, we define the Neumann Laplacian
	\[
	A := A_p = -\Delta
	\]
	with domain
	\[
	D(A)
	:=
	\left\{
	u \in W^{2,p}(\mathcal{O})
	\;:\;
	\frac{\partial u}{\partial \nu}=0
	\text{ on } \partial\mathcal{O}
	\right\}.
	\]
	The spectrum of $A$ is a $p$-independent countable set of nonnegative real numbers
	\[
	0=\nu_0<\nu_1<\nu_2<\cdots.
	\]
	For $\gamma\ge 0$, the fractional power $(A+1)^\gamma$ is well defined and satisfies the embeddings
	\[
	D((A_p+1)^\gamma)
	\hookrightarrow W^{1,p}(\mathcal{O})
	\qquad \text{if } \gamma>\frac12,
	\]
	and
	\[
	D((A_p+1)^\gamma)
	\hookrightarrow C(\overline{\mathcal O})
	\qquad \text{if } 2\gamma>\frac{n}{p}.
	\]
	The following smoothing estimates are classical (see, e.g., \cite{HORSTMANN200552}).

	\begin{lemma}[Heat semigroup estimates]
		\label{lem:heat-semigroup}
		There exist constants $C>0$ and $\nu>0$, smaller than the first nonzero eigenvalue $\nu_1$, such that:
		\begin{enumerate}
			\item[(i)] For all $t>0$, $1<p<\infty$, $\gamma\ge0$, and $\phi\in L^p(\mathcal{O})$,
			\[
			\|(A_p+1)^\gamma e^{-tA}\phi\|_{L^p(\mathcal{O})}
			\le
			C(1+t^{-\gamma})
			\|\phi\|_{L^p(\mathcal{O})}.
			\]

			\item[(ii)] For all $t>0$, $1<p<\infty$, and $\phi\in L^p(\mathcal{O})$,
			\[
			\|\nabla e^{-tA}\phi\|_{L^p(\mathcal{O})}
			\le
			C(1+t^{-1/2})
			\|\phi\|_{L^p(\mathcal{O})}.
			\]

			\item[(iii)] If $\gamma>0$, $p\in(1,\infty)$, and $t>0$, then
			\[
			\|(A_p+1)^\gamma e^{-tA}\nabla\cdot\boldsymbol\phi\|_{L^p(\mathcal{O})}
			\le
			C(1+t^{-\gamma-\frac12-\varepsilon})
			\|\boldsymbol\phi\|_{L^p(\mathcal{O};\mathbb R^n)},
			\qquad t>0,
			\]
			for $\boldsymbol\phi\in L^p(\mathcal{O};\mathbb R^n)$, where $\varepsilon>0$ can be taken arbitrarily small.

			\item[(iv)] For $w\in L^p(\mathcal{O})$,
			\[
			\|(A+1)^\gamma e^{-t(A+1)}w\|_{L^p(\mathcal{O})}
			\le
			Ct^{-\gamma}e^{-\nu t}\|w\|_p,
			\qquad t>0.
			\]
		\end{enumerate}
	\end{lemma}

	Let $(\Omega,\mathcal{F},(\mathcal{F}_t)_{t\ge 0},\mathbb{P})$ be a filtered probability space. For each $i=1,2$, let $(W_k^i)_{k\ge1}$ be a sequence of independent real-valued Brownian motions. We write
	\[
	\sigma_i(z)\,dW_i(t)
	:=
	\sum_{k=1}^{\infty}\sigma_k^i(z)\,dW_k^i(t)
	\]
	and impose the following assumption.
	\begin{assumption}\label{ass:stochastic-sigma}
		For each $i=1,2$, let $\sigma^i=(\sigma_k^i)_{k\ge1}:\mathbb R\to\ell^2$. There exists $K>0$ such that, for all $z,z_1,z_2\in\mathbb R$,
		\begin{enumerate}
			\item $\sigma^i(0)  = 0$,
			\item $\|\sigma^i(z)\|_{\ell^2}  \le K(|z| + 1)$,
			\item $\|\sigma^i(z_1) - \sigma^i(z_2)\|_{\ell^2} \le K|z_1-z_2|$.
		\end{enumerate}
	\end{assumption}

	We recall a lemma which provides estimates for the convolution \(M(t) = \int_{0}^{t}e^{-(t-s)A}h(s)\,dW_s\).
	\begin{lemma}\label{lem:stochastic-convolution}
		[{\cite[Lemma~3.1]{CHEN2025113531}}] Suppose $h=(h_k)_{k=1}^\infty$ is an $\ell^2$-valued progressively measurable function on $\Omega\times[0,T]\times\mathcal O$ such that \[\mathbb{E}\|\|h\|_{\ell^2}\|_{2,2,T}^2<\infty.\] Then the stochastic convolution has the following $L^\infty$ estimate:
		\begin{equation}
			\mathbb{E}\sup_{t\in[0,T\wedge\tau]} \|M(t)\|_{L^\infty(\mathcal{O})}^\eta
			\le C (1+T)^{C'}
			\left(
			\mathbb{E}\|\|h\|_{\ell^2}\|_{2r,2q,T\wedge\tau}^{\eta}
			\right),
		\end{equation}
		for any $\eta > 0$, $r, q \in (1, \infty]$ with $\frac{1}{r} + \frac{1}{q} < 1 $ and stopping time $\tau$. The constants $C, C'$  depend on $\eta, r, q$.

		In particular, the estimate will be mainly used in the case $\eta=1$, $r=\frac32$ and $q=\infty$. Then it becomes
		\begin{equation}\label{eq:convolution-infty}
			\begin{aligned}
		\mathbb E\sup_{t\in[0,T\wedge\tau]}
		\|M(t)\|_{L^\infty(\mathcal O)}
		&\le
		C(1+T)^{C'}
		\mathbb E
		\left(
		\int_0^{T\wedge\tau}
		\|\|h(s)\|_{\ell^2}\|_{L^\infty(\mathcal O)}^3\,ds
		\right)^{\frac13}
		\end{aligned}
	\end{equation}
	\end{lemma}

	\begin{assumption}\label{ass:initial-data}
		The initial data $u_0$, $v_0$ and $w_0$ are $\mathcal F_0$-measurable random variables such that
		\[
		u_0,v_0\in C(\overline{\mathcal O}),
		\qquad
		w_0\in W^{1,\infty}(\mathcal O),
		\quad \mathbb P\text{-a.s.}
		\]
		Moreover,
		\[
		u_0(x)\ge0,\qquad v_0(x)\ge0,\qquad w_0(x)\ge0,
		\quad x\in{\mathcal O},
		\quad \mathbb P\text{-a.s.}
		\]
		Finally, for every $p\ge2$,
		\[
		\mathbb E\left[
		\left(
		\|u_0\|_{L^\infty}
		+\|v_0\|_{L^\infty}
		+\|w_0\|_{W^{1,\infty}}
		\right)^p
		\right]<\infty.
		\]
	\end{assumption}

	Throughout the paper, $C$ denotes a generic positive constant which may change from line to line.




	\section{Local Existence of Solutions}

	In this section, we prove the existence and uniqueness of a maximal local mild solution. We use only $\zeta_i,\eta_i\ge1$ here; the stronger self-damping condition is imposed in Section~4. We first solve a truncated system by the Banach fixed point theorem, then remove the truncation by stopping times. We also prove that the solution remains nonnegative.


	\subsection{Existence of truncated system}

	Let $\theta\in C^1([0,\infty);[0,1])$ satisfy
	\[
	\theta(r)=1\quad\text{for }0\le r\le1,
	\qquad
	\theta(r)=0\quad\text{for }r\ge2,
	\qquad
	\sup_{r\ge0}|\theta'(r)|<\infty.
	\]
	For $m\in\mathbb N$, set
	\[
	\theta_m(r):=\theta\left(\frac r m\right).
	\]
	For a pair $U=(u,v)$, define
	\[
	R_t(U)
	:=
	\sup_{0\le s\le t}
	\left(
	\|u(s)\|_{L^\infty}
	+
	\|v(s)\|_{L^\infty}
	\right),
	\qquad
	\Theta_m^U(t):=\theta_m\bigl(R_t(U)\bigr).
	\]


	We consider the truncated mild system that for every $m\in\mathbb N$ and every $T>0$
	\begin{equation}
		\label{eq:cutoff-system}
		\left\{
		\begin{aligned}
			d u
			={}&
			\Big[
			\Delta u
			-\chi_1\Theta_m^U(t)\nabla\cdot(u\nabla w)
			+\Theta_m^U(t)F_1(u,v)
			\Big]dt
			+\sigma_1(u)\,dW_1(t),
			\\
			d v
			={}&
			\Big[
			\Delta v
			-\chi_2\Theta_m^U(t)\nabla\cdot(v\nabla w)
			+\Theta_m^U(t)F_2(v,u)
			\Big]dt
			+\sigma_2(v)\,dW_2(t),
			\\
			\partial_t w
			={}&
			\Delta w-w+\alpha u+\beta v,
		\end{aligned}
		\right.
	\end{equation}
	with homogeneous Neumann boundary condition
	\[
	\partial_\nu u=\partial_\nu v=\partial_\nu w=0
	\quad\text{on }\partial\mathcal O,
	\]
	and initial condition
	\[
	u(0)=u_0,\qquad v(0)=v_0,\qquad w(0)=w_0,
	\]
	where
	\[
	F_1(u,v)
	=
	\mu_1 u\left(1-u_+^{\zeta_1}-a_1v_+^{\eta_1}\right),
	\qquad
	F_2(v,u)
	=
	\mu_2 v\left(1-v_+^{\zeta_2}-a_2u_+^{\eta_2}\right),
	\]
	and $r_+=\max\{r,0\}$.

	\begin{theorem}[Cut-off system]
		\label{thm:cutoff-existence}
		Suppose that Assumptions~\ref{ass:stochastic-sigma} and~\ref{ass:initial-data} hold. For every $T>0$ and $q>n$, the cut-off system~\eqref{eq:cutoff-system} admits a unique mild solution satisfying
		\[
		u,v\in
		L^1\bigl(
		\Omega;
		C([0,T];C(\overline{\mathcal O}))
		\bigr),
		\]
		and
		\[
		w\in
		L^1\bigl(
		\Omega;
		C([0,T];W^{1,q}(\mathcal O))
		\bigr).
		\]
		Moreover, $u,v,w\ge0$ almost surely.
	\end{theorem}

	\begin{proof}
		For $T>0$, define $S_T$ as the space of all progressively measurable $C(\overline{\mathcal O})\times C(\overline{\mathcal O})$-valued processes $U=(u,v)$ such that
		\[
		\|U\|_{S_T}
		:=
		\mathbb E
		\sup_{0\le t\le T}
		\left(
		\|u(t)\|_{L^\infty}
		+
		\|v(t)\|_{L^\infty}
		\right)
		<\infty.
		\]
		Then $S_T$, equipped with the above norm, is a Banach space.

		For each $U=(u,v)\in S_T$, let $W_U$ be the solution of
		\[
		\partial_t w
		=
		\Delta w - w +\alpha u+\beta v,
		\qquad
		{w}_\nu = 0 \text{ on } \partial \mathcal{O},
		\qquad
		w(0)=w_0.
		\]


		Fix $q>n$ and $\gamma'\in(1/2,1)$. Using the embedding
		\[
		D\bigl((A_q+1)^{\gamma'}\bigr)\hookrightarrow W^{1,q}(\mathcal O)
		\]
		and the heat-semigroup estimates (see lemma \ref{lem:heat-semigroup}), we obtain, for $0\le t\le T$,
		\begin{equation}
		\label{eq:WU-estimate-q}
		\begin{aligned}
			\|W_U(t)\|_{W^{1,q}} &\leq \left\|e^{-t(A+1)}w_0\|_{W^{1,\infty}} + \int_{0}^{t}  \|(A+1)^{\gamma'} e^{-(t-s)(A+1)} (\alpha u +\beta v)\right\|_{L^q}\,ds\\
							  &\leq \|e^{-t(A+1)}w_0\|_{W^{1,\infty}} + \int_{0}^{t}  (t-s)^{-\gamma'} e^{-\nu(t-s)}\| \alpha u +\beta v\|_{L^q}\,ds\\
							  &\leq C_{\|w_0\|_{W^{1,\infty}}} + C T^{1-\gamma'} (\sup_{0\le s\le t}\|u(s)\|_{L^q} + \sup_{0\le s\le t}\|v(s)\|_{L^q}),
		\end{aligned}
		\end{equation}
		and
		\begin{equation}
			\label{eq:WU-estimate-Linf}
			\begin{aligned}
				\|W_U(t)\|_{W^{1,q}}
				&\le
				C_{w_0} + C T^{1-\gamma'}
				\left(
				\sup_{0\le s\le t}\|u(s)\|_{L^\infty}
				+
				\sup_{0\le s\le t}\|v(s)\|_{L^\infty}
				\right) .
			\end{aligned}
		\end{equation}
		For later use, we also record the difference estimate. For $U_i=(u_i,v_i)\in S_T$, $i=1,2$, we have
		\begin{equation}
			\label{eq:WU-difference}
			\begin{aligned}
				\|W_{U_1}(t)-W_{U_2}(t)\|_{W^{1,q}} \le
				C T^{1-\gamma'}
				\sup_{0\le s\le t}
				\left(
				\|u_1(s)-u_2(s)\|_{L^\infty}
				+
				\|v_1(s)-v_2(s)\|_{L^\infty}
				\right).
			\end{aligned}
		\end{equation}

		We now define a map
		\[
		\Phi=(\Phi_1,\Phi_2):S_T\to S_T
		\]
		by
		\begin{equation}
			\label{eq:Phi1}
			\begin{aligned}
				\Phi_1(U)(t)
				={}&
				e^{-tA}u_0
				-\chi_1
				\int_0^t
				e^{-(t-s)A}\nabla\cdot
				\left(
				\Theta_m^U(s)u(s)\nabla W_U(s)
				\right)\,ds
				\\
				&+
				\int_0^t
				e^{-(t-s)A}\Theta_m^U(s)F_1(u(s),v(s))\,ds
				\\
				&+
				\int_0^t
				e^{-(t-s)A}\sigma_1(u(s))\,dW_1(s),
			\end{aligned}
		\end{equation}
		and
		\begin{equation}
			\label{eq:Phi2}
			\begin{aligned}
				\Phi_2(U)(t)
				={}&
				e^{-tA}v_0
				-\chi_2
				\int_0^t
				e^{-(t-s)A}\nabla\cdot
				\left(
				\Theta_m^U(s)v(s)\nabla W_U(s)
				\right)\,ds
				\\
				&+
				\int_0^t
				e^{-(t-s)A}\Theta_m^U(s)F_2(v(s),u(s))\,ds
				\\
				&+
				\int_0^t
				e^{-(t-s)A}\sigma_2(v(s))\,dW_2(s).
			\end{aligned}
		\end{equation}

	We next show that $\Phi$ maps $S_T$ into itself. Choose $p>2$, use $q=p$ in the preceding estimates, and take
	\[
	\frac1p<\gamma<\frac12,
	\qquad
	0<\varepsilon<\frac12-\gamma.
	\]
	By Lemma~\ref{lem:heat-semigroup},
	\begin{equation}
		\begin{aligned}
			\|\Phi_1(U)(t)\|_{L^\infty}
			&\le
			C\|e^{-tA}u_0\|_{L^\infty}
			+
			\chi_1
			\int_0^t \left\|(A+1)^\gamma e^{-(t-s)A}\nabla\cdot
			\left(
			\Theta_m^U(s)u(s)\nabla W_U(s)
			\right)\right\|_{L^p}\,ds
			\\
			&\quad+
			\int_0^t
			\left\|e^{-(t-s)A}\Theta_m^U(s)F_1(u(s),v(s))\right\|_{L^\infty}\,ds\\
			&\quad
			+
			\left\|
			\int_0^t
			e^{-(t-s)A}\sigma_1(u(s))\,dW_1(s)
			\right\|_{L^\infty}
			\\
			&\le
			C\|u_0\|_{L^\infty}
			+
			C \int_0^t (t-s)^{- \frac{1}{2} - \gamma -\varepsilon} \left\|
			\left(
			\Theta_m^U(s)u(s)\nabla W_U(s)
			\right)\right\|_{L^p}\,ds  \\
			&\quad+
			\int_0^t
			\left\|\Theta_m^U(s)F_1(u(s),v(s))\right\|_{L^\infty}\,ds
			\\
			&\quad
			+
			\left\|
			\int_0^t
			e^{-(t-s)A}\sigma_1(u(s))\,dW_1(s)
			\right\|_{L^\infty}.
		\end{aligned}
	\end{equation}

	On the support of $\Theta_m^U(s)$, one has
	\[
	\|u(s)\|_{L^\infty}
	+
	\|v(s)\|_{L^\infty}
	\le2m.
	\]
	Combining this with \eqref{eq:WU-estimate-Linf}, for $p>n$, we get
	\begin{equation*}
		\left\|
		\Theta_m^U(s)u(s)\nabla W_U(s)
		\right\|_{L^p}
		+
		\left\|
		\Theta_m^U(s)v(s)\nabla W_U(s)
		\right\|_{L^p} \le C_m \|\nabla W_U\|_{L^p}
		\le C_{m}\bigl(1+T^{1-\gamma'}\bigr) .
	\end{equation*}
	Moreover, on the same support,
	\begin{equation*}
		\left\|
		\Theta_m^U(s)F_1(u(s),v(s))
		\right\|_{L^\infty}
		+
		\left\|
		\Theta_m^U(s)F_2(v(s),u(s))
		\right\|_{L^\infty}
		\le
		C_{m}.
	\end{equation*}
	Hence
	\begin{equation}
		\label{eq:Phi1-selfmap-pathwise}
		\begin{aligned}
			\|\Phi_1(U)(t)\|_{L^\infty}	
			&\le
			C\|u_0\|_{L^\infty}
			+
			C_{m}\bigl(1+T^{1-\gamma'}\bigr)t^{\frac{1}{2}-\gamma-\varepsilon}
			+
			C_m t
			\\
			&\quad
			+
			\left\|
			\int_0^t
			e^{-(t-s)A}\sigma_1(u(s))\,dW_1(s)
			\right\|_{L^\infty}.
		\end{aligned}
	\end{equation}
	Taking the supremum over $t\in[0,T]$ and then taking expectations, by Lemma~\ref{lem:stochastic-convolution} and the linear growth of $\sigma_1$, we have
	\begin{equation}
		\label{eq:stoch-selfmap-u}
		\mathbb E
		\sup_{0\le t\le T}
		\left\|
		\int_0^t
		e^{-(t-s)A}\sigma_1(u(s))\,dW_1(s)
		\right\|_{L^\infty}
		\le
		C_T
		\left(
		1+
		\mathbb E
		\sup_{0\le t\le T}
		\|u(t)\|_{L^\infty}
		\right).
	\end{equation}
	Thus
	\[
	\mathbb E
	\sup_{0\le t\le T}
	\|\Phi_1(U)(t)\|_{L^\infty}
	<\infty.
	\]
	The same argument yields
	\[
	\mathbb E
	\sup_{0\le t\le T}
	\|\Phi_2(U)(t)\|_{L^\infty}
	<\infty.
	\]
	Hence $\Phi(U)\in S_T$.

	It remains to prove that $\Phi$ is a contraction. Let $U_i=(u_i,v_i)\in S_T$, $i=1,2$. The difference between \(\Phi(U_1)\) and \(\Phi(U_2)\) is
	\begin{equation}
		\label{eq:Phi-difference}
		\begin{aligned}
			\Phi_1(U_1)(t) - \Phi_1(U_2)(t)
			={}&
			-\chi_1
			\int_0^t
			e^{-(t-s)A}\nabla\cdot\Bigl(
			\Theta_m^{U_1}(s)u_1(s)\nabla W_{U_1}(s)
			\\[-0.2em]
			&-\Theta_m^{U_2}(s)u_2(s)\nabla W_{U_2}(s)\Bigr)ds
			\\
			&+
			\int_0^t
			e^{-(t-s)A}\Bigl(\Theta_m^{U_1}(s)F_1(u_1(s),v_1(s))
			\\[-0.2em]
			&-\Theta_m^{U_2}(s)F_1(u_2(s),v_2(s))\Bigr) \,ds
			\\
			&+
			\int_0^t
			e^{-(t-s)A}
			\Bigl(\sigma_1(u_1(s))-\sigma_1(u_2(s))\Bigr)
			\,dW_1(s).
		\end{aligned}
	\end{equation}
	Since
	\[
	\theta_m'(r)
	=
	\frac1m\theta'\left(\frac r m\right),
	\]
	we have
	\[
	|\theta_m'(r)|\le \frac{C}{m}.
	\]
	By the definition,
	\[
	\begin{aligned}
		&\quad
		\left|
		R_t(U_1)-R_t(U_2)
		\right|
		\\
		&\le
		\sup_{0\le s\le t}
		\left(
		\|u_1(s)-u_2(s)\|_{L^\infty}
		+
		\|v_1(s)-v_2(s)\|_{L^\infty}
		\right).
	\end{aligned}
	\]
	Therefore,
	\begin{equation}
		\label{eq:cutoff-difference}
		\begin{aligned}
			&\quad
			\left|
			\Theta_m^{U_1}(t)-\Theta_m^{U_2}(t)
			\right|
			\\
			&\le
			\frac{C}{m}
			\sup_{0\le s\le t}
			\left(
			\|u_1(s)-u_2(s)\|_{L^\infty}
			+
			\|v_1(s)-v_2(s)\|_{L^\infty}
			\right).
		\end{aligned}
	\end{equation}

	We first estimate the reaction terms. We claim that
	\begin{equation}
		\label{eq:F1-difference}
		\begin{aligned}
			&\quad
			\left\|
			\Theta_m^{U_1}(t)F_1(u_1(t),v_1(t))
			-
			\Theta_m^{U_2}(t)F_1(u_2(t),v_2(t))
			\right\|_{L^\infty}
			\\
			&\le
			C_m
			\sup_{0\le s\le t}
			\left(
			\|u_1(s)-u_2(s)\|_{L^\infty}
			+
			\|v_1(s)-v_2(s)\|_{L^\infty}
			\right).
		\end{aligned}
	\end{equation}
	Indeed, if
	\[
	R_t(U_1)\vee R_t(U_2)\le2m,
	\]
	then all arguments of $F_1$ remain in a bounded interval depending only on $m$. Hence, by the mean value theorem and \eqref{eq:cutoff-difference},
	\[
	\begin{aligned}
		&\quad
		\left\|
		\Theta_m^{U_1}F_1(u_1,v_1)
		-
		\Theta_m^{U_2}F_1(u_2,v_2)
		\right\|_{L^\infty}
		\\
		&\le
		\left|
		\Theta_m^{U_1}
		-
		\Theta_m^{U_2}
		\right|
		\|F_1(u_1,v_1)\|_{L^\infty}
		+
		\Theta_m^{U_2}
		\|F_1(u_1,v_1)-F_1(u_2,v_2)\|_{L^\infty}
		\\
		&\le
		C_m
		\sup_{0\le s\le t}
		\left(
		\|u_1(s)-u_2(s)\|_{L^\infty}
		+
		\|v_1(s)-v_2(s)\|_{L^\infty}
		\right).
	\end{aligned}
	\]
	If
	\[
	R_t(U_1)\wedge R_t(U_2)\ge2m,
	\]
	then both cut-off factors vanish. If, for instance,
	\[
	R_t(U_1)>2m,
	\qquad
	R_t(U_2)\le2m,
	\]
	then $\Theta_m^{U_1}(t)=0$ and, by \eqref{eq:cutoff-difference},
	\[
	\Theta_m^{U_2}(t)
	=
	\left|
	\Theta_m^{U_2}(t)-\Theta_m^{U_1}(t)
	\right|
	\le
	\frac{C}{m}
	\sup_{0\le s\le t}
	\left(
	\|u_1(s)-u_2(s)\|_{L^\infty}
	+
	\|v_1(s)-v_2(s)\|_{L^\infty}
	\right).
	\]
	On the set where $\Theta_m^{U_2}(t)\ne0$, one has
	\[
	\|F_1(u_2(t),v_2(t))\|_{L^\infty}\le C_m.
	\]
	This proves \eqref{eq:F1-difference}. The other mixed case is identical. Similarly,
	\begin{equation}
		\label{eq:F2-difference}
		\begin{aligned}
			&
			\left\|
			\Theta_m^{U_1}(t)F_2(v_1(t),u_1(t))
			-
			\Theta_m^{U_2}(t)F_2(v_2(t),u_2(t))
			\right\|_{L^\infty}
			\\
			&\qquad\le
			C_m
			\sup_{0\le s\le t}
			\left(
			\|u_1(s)-u_2(s)\|_{L^\infty}
			+
			\|v_1(s)-v_2(s)\|_{L^\infty}
			\right).
		\end{aligned}
	\end{equation}

	We now estimate the chemotaxis terms. More precisely,  we show that
	\begin{equation}
		\label{eq:chemotaxis-u-difference}
		\begin{aligned}
			&
			\left\|
			\Theta_m^{U_1}(t)u_1(t)\nabla W_{U_1}(t)
			-
			\Theta_m^{U_2}(t)u_2(t)\nabla W_{U_2}(t)
			\right\|_{L^p}
			\\
			&\qquad\le
			C_{m}\bigl(1+T^{1-\gamma'}\bigr)
			\sup_{0\le s\le t}
			\left(
			\|u_1(s)-u_2(s)\|_{L^\infty}
			+
			\|v_1(s)-v_2(s)\|_{L^\infty}
			\right).
		\end{aligned}
	\end{equation}
	If
	\[
	R_t(U_1)\vee R_t(U_2)\le2m,
	\]
	then, by \eqref{eq:WU-estimate-Linf}, \eqref{eq:WU-difference}, and \eqref{eq:cutoff-difference},
	\[
	\begin{aligned}
		&\quad
		\left\|
		\Theta_m^{U_1}u_1\nabla W_{U_1}
		-
		\Theta_m^{U_2}u_2\nabla W_{U_2}
		\right\|_{L^p}
		\\
		&\le
		\left|
		\Theta_m^{U_1}
		-
		\Theta_m^{U_2}
		\right|
		\|u_1\nabla W_{U_1}\|_{L^p}
		\\
		&\quad
		+
		\Theta_m^{U_2}
		\|(u_1-u_2)\nabla W_{U_1}\|_{L^p}
		+
		\Theta_m^{U_2}
		\|u_2\nabla(W_{U_1}-W_{U_2})\|_{L^p}
		\\
		&\le
		C_{m}\bigl(1+T^{1-\gamma'}\bigr)
		\sup_{0\le s\le t}
		\left(
		\|u_1(s)-u_2(s)\|_{L^\infty}
		+
		\|v_1(s)-v_2(s)\|_{L^\infty}
		\right).
	\end{aligned}
	\]
	If
	\[
	R_t(U_1)\wedge R_t(U_2)\ge2m,
	\]
	then both cut-off factors vanish. If, for instance,
	\[
	R_t(U_1)>2m,
	\qquad
	R_t(U_2)\le2m,
	\]
	then $\Theta_m^{U_1}(t)=0$. Hence, using \eqref{eq:cutoff-difference} and \eqref{eq:WU-estimate-Linf},
	\[
	\begin{aligned}
		&\quad
		\left\|
		\Theta_m^{U_2}(t)u_2(t)\nabla W_{U_2}(t)
		\right\|_{L^p}
		\\
		&\le
		\left|
		\Theta_m^{U_2}(t)-\Theta_m^{U_1}(t)
		\right|
		\|u_2(t)\nabla W_{U_2}(t)\|_{L^p}
		\\
		&\le
		C_{m}\bigl(1+T^{1-\gamma'}\bigr)
		\sup_{0\le s\le t}
		\left(
		\|u_1(s)-u_2(s)\|_{L^\infty}
		+
		\|v_1(s)-v_2(s)\|_{L^\infty}
		\right).
	\end{aligned}
	\]
	The other mixed case is identical. Thus \eqref{eq:chemotaxis-u-difference} follows. The same argument gives
	\begin{equation}
		\label{eq:chemotaxis-v-difference}
		\begin{aligned}
			&\quad
			\left\|
			\Theta_m^{U_1}(t)v_1(t)\nabla W_{U_1}(t)
			-
			\Theta_m^{U_2}(t)v_2(t)\nabla W_{U_2}(t)
			\right\|_{L^p}
			\\
			&\le
			C_{m}\bigl(1+T^{1-\gamma'}\bigr)
			\sup_{0\le s\le t}
			\left(
			\|u_1(s)-u_2(s)\|_{L^\infty}
			+
			\|v_1(s)-v_2(s)\|_{L^\infty}
			\right).
		\end{aligned}
	\end{equation}

	By Lemma~\ref{lem:heat-semigroup} and \eqref{eq:chemotaxis-u-difference},
	\begin{equation}
		\label{eq:semigroup-chemotaxis-difference-u}
		\begin{aligned}
			&
			\sup_{0\le t\le T}
			\left\|
			\int_0^t
			e^{-(t-s)A}\nabla\cdot
			\left[
			\Theta_m^{U_1}(s)u_1(s)\nabla W_{U_1}(s)
			-
			\Theta_m^{U_2}(s)u_2(s)\nabla W_{U_2}(s)
			\right]ds
			\right\|_{L^\infty}
			\\
			&\quad\le
			C_{m}\bigl(1+T^{1-\gamma'}\bigr)T^{\frac{1}{2}-\gamma-\varepsilon}
			\sup_{0\le s\le T}
			\left(
			\|u_1(s)-u_2(s)\|_{L^\infty}
			+
			\|v_1(s)-v_2(s)\|_{L^\infty}
			\right).
		\end{aligned}
	\end{equation}
	The same estimate follows from \eqref{eq:chemotaxis-v-difference} for the $v$-equation.

	For the reaction terms, \eqref{eq:F1-difference} gives
	\begin{equation}
		\label{eq:semigroup-reaction-difference-u}
		\begin{aligned}
			&
			\sup_{0\le t\le T}
			\left\|
			\int_0^t
			e^{-(t-s)A}
			\left[
			\Theta_m^{U_1}(s)F_1(u_1(s),v_1(s))
			-
			\Theta_m^{U_2}(s)F_1(u_2(s),v_2(s))
			\right]ds
			\right\|_{L^\infty}
			\\
			&\quad\le
			C_m T
			\sup_{0\le s\le T}
			\left(
			\|u_1(s)-u_2(s)\|_{L^\infty}
			+
			\|v_1(s)-v_2(s)\|_{L^\infty}
			\right).
		\end{aligned}
	\end{equation}
	Similarly,
	\begin{equation}
		\label{eq:semigroup-reaction-difference-v}
		\begin{aligned}
			&
			\sup_{0\le t\le T}
			\left\|
			\int_0^t
			e^{-(t-s)A}
			\left[
			\Theta_m^{U_1}(s)F_2(v_1(s),u_1(s))
			-
			\Theta_m^{U_2}(s)F_2(v_2(s),u_2(s))
			\right]ds
			\right\|_{L^\infty}
			\\
			&\quad\le
			C_m T
			\sup_{0\le s\le T}
			\left(
			\|u_1(s)-u_2(s)\|_{L^\infty}
			+
			\|v_1(s)-v_2(s)\|_{L^\infty}
			\right).
		\end{aligned}
	\end{equation}

	Finally, Lemma~\ref{lem:stochastic-convolution} yields that for some $\rho>0$,
	\begin{equation}
		\label{eq:stoch-difference-u}
		\begin{aligned}
			&
			\mathbb E
			\sup_{0\le t\le T}
			\left\|
			\int_0^t
			e^{-(t-s)A}
			\bigl[
			\sigma_1(u_1(s))-\sigma_1(u_2(s))
			\bigr]\,dW_1(s)
			\right\|_{L^\infty}
			\\
			&\quad\le
			CT^\rho
			\mathbb E
			\sup_{0\le t\le T}
			\|u_1(t)-u_2(t)\|_{L^\infty}.
		\end{aligned}
	\end{equation}
	Likewise,
	\begin{equation}
		\label{eq:stoch-difference-v}
		\begin{aligned}
			&
			\mathbb E
			\sup_{0\le t\le T}
			\left\|
			\int_0^t
			e^{-(t-s)A}
			\bigl[
			\sigma_2(v_1(s))-\sigma_2(v_2(s))
			\bigr]\,dW_2(s)
			\right\|_{L^\infty}
			\\
			&\quad\le
			CT^\rho
			\mathbb E
			\sup_{0\le t\le T}
			\|v_1(t)-v_2(t)\|_{L^\infty}.
		\end{aligned}
	\end{equation}

	Combining \eqref{eq:semigroup-chemotaxis-difference-u}--\eqref{eq:stoch-difference-v}, we find
	\begin{equation}
		\label{eq:Phi-contraction}
		\|\Phi(U_1)-\Phi(U_2)\|_{S_T}
		\le
		C_{m}
		\left(
		\bigl(1+T^{1-\gamma'}\bigr)T^{\frac{1}{2}-\gamma-\varepsilon}+T+T^\rho
		\right)
		\|U_1-U_2\|_{S_T}.
	\end{equation}
	Choose $T_0>0$ sufficiently small such that
	\[
	C_{m}
	\left(
	\bigl(1+T_0^{1-\gamma'}\bigr)T_0^{\frac{1}{2}-\gamma-\varepsilon}+T_0+T_0^\rho
	\right)
	\le
	\frac12.
	\]
	Then $\Phi$ is a contraction on $S_{T_0}$. By Banach's fixed point theorem, there exists a unique $U=(u,v)\in S_{T_0}$ such that
	\[
	\Phi(U)=U.
	\]
	The fixed point has the mild formulation
	\begin{equation*}
		\begin{aligned}
			u(t)
			={}&
			e^{-tA}u_0
			-\chi_1
			\int_0^t
			e^{-(t-s)A}\nabla\cdot
			\left(
			\Theta_m^U(s)u(s)\nabla w(s)
			\right)ds
			\\
			&+
			\int_0^t
			e^{-(t-s)A}\Theta_m^U(s)F_1(u(s),v(s))\,ds
			+
			\int_0^t
			e^{-(t-s)A}\sigma_1(u(s))\,dW_1(s),
		\end{aligned}
	\end{equation*}
	\begin{equation*}
		\begin{aligned}
			v(t)
			={}&
			e^{-tA}v_0
			-\chi_2
			\int_0^t
			e^{-(t-s)A}\nabla\cdot
			\left(
			\Theta_m^U(s)v(s)\nabla w(s)
			\right)ds
			\\
			&+
			\int_0^t
			e^{-(t-s)A}\Theta_m^U(s)F_2(v(s),u(s))\,ds
			+
			\int_0^t
			e^{-(t-s)A}\sigma_2(v(s))\,dW_2(s),
		\end{aligned}
	\end{equation*}
	and
	\begin{equation*}
		w(t)
		=
		e^{-t(A+1)}w_0
		+
		\int_0^t
		e^{-(t-s)(A+1)}
		\bigl(
		\alpha u(s)+\beta v(s)
		\bigr)\,ds .
	\end{equation*}
	Hence $(u,v,w)$ is the unique mild solution of the cut-off system on $[0,T_0]$.

	The solution can be extended to any finite interval $[0,T]$ by the standard stepwise argument. Indeed, after solving the system on $[0,T_0]$, we repeat the same fixed point argument on $[T_0,2T_0]$, with $u(T_0),v(T_0),w(T_0)$ as initial data and with the cut-off factor still defined by the full past supremum
	\[
	\Theta_m^U(t)
	=
	\theta_m
	\left(
	\sup_{0\le s\le t}
	\left(
	\|u(s)\|_{L^\infty}
	+
	\|v(s)\|_{L^\infty}
	\right)
	\right).
	\]
	Whenever this cut-off factor is nonzero, the full supremum above is bounded by $2m$, so the same estimates remain valid on each short interval. Repeating the procedure finitely many times yields a unique mild solution on $[0,T]$.

	Finally, suppose that $u_0,v_0\ge0$. We prove that the solution is nonnegative. It is enough to prove the assertion for $u$, since the proof for $v$ is the same.

	For $r\in\mathbb R$, set
	\[
	r^+ := \max\{r,0\},
	\qquad
	r^- := \max\{-r,0\}.
	\]
	Let $\delta\in(0,1)$ and define
	\[
	g_\delta(r):=\frac{r^2}{\delta+r},
	\qquad r\in[0,\infty),
	\]
	and
	\[
	G_\delta(r):=g_\delta\bigl((r^-)^2\bigr),
	\qquad r\in\mathbb R.
	\]
	Then $G_\delta\in C^2(\mathbb R)$,
	\[
	G_\delta(r)=G_\delta'(r)=G_\delta''(r)=0,
	\qquad r\ge0,
	\]
	and there exists a constant $C>0$, independent of $\delta$, such that
	\[
	|G_\delta'(r)|\le C r^-,
	\qquad
	0\le G_\delta''(r)\le C,
	\qquad r\in\mathbb R.
	\]
	Moreover, as $\delta\to0$,
	\[
	G_\delta(r)\to (r^-)^2,
	\qquad
	G_\delta'(r)\to -2r^-,
	\qquad
	G_\delta''(r)\to 2\cdot\mathbf 1_{\{r<0\}}
	\]
	for a.e. $r\in\mathbb R$.

	Define
	\[
	\phi_\delta(f):=\int_{\mathcal O}G_\delta(f(x))\,dx,
	\qquad f\in L^2(\mathcal O).
	\]
	Then $\phi_\delta$ is twice G\^ateaux differentiable on $L^2(\mathcal O)$, and
	\[
	D\phi_\delta(f)h
	=
	\int_{\mathcal O}G_\delta'(f)h\,dx,
	\qquad
	D^2\phi_\delta(f)(h,h)
	=
	\int_{\mathcal O}G_\delta''(f)h^2\,dx.
	\]

	Since the solution is mild, we apply It\^o's formula to the Yosida approximations and then pass to the limit. After taking expectations, the martingale term vanishes and the resulting identity is
	\[
	\begin{aligned}
		&\quad\mathbb E\phi_\delta(u(t))
		+
		\mathbb E
		\int_0^t
		\int_{\mathcal O}
		G_\delta''(u)|\nabla u|^2\,dx\,ds
		\\
		&=
		\phi_\delta(u_0)
		+
		\chi_1
		\mathbb E
		\int_0^t
		\Theta_m^U(s)
		\int_{\mathcal O}
		G_\delta''(u)u\nabla u\cdot\nabla w\,dx\,ds
		\\
		&\quad
		+
		\mathbb E
		\int_0^t
		\Theta_m^U(s)
		\int_{\mathcal O}
		G_\delta'(u)F_1(u,v)\,dx\,ds
		\\
		&\quad
		+
		\frac12
		\mathbb E
		\int_0^t
		\int_{\mathcal O}
		G_\delta''(u)\|\sigma_1(u)\|_{\ell^2}^2\,dx\,ds .
	\end{aligned}
	\]
	Here we have used the homogeneous Neumann boundary condition to obtain
	\[
	\int_{\mathcal O}G_\delta'(u)\Delta u\,dx
	=
	-\int_{\mathcal O}G_\delta''(u)|\nabla u|^2\,dx,
	\]
	and
	\[
	-\int_{\mathcal O}G_\delta'(u)\nabla\cdot(u\nabla w)\,dx
	=
	\int_{\mathcal O}G_\delta''(u)u\nabla u\cdot\nabla w\,dx.
	\]
	Since $u_0\ge0$, we have
	\[
	\phi_\delta(u_0)=0.
	\]

	Letting $\delta\to0$ and using the dominated convergence theorem, we get
	\[
	\begin{aligned}
		&\quad\mathbb E\|u^-(t)\|_{L^2}^2
		+
		2\mathbb E
		\int_0^t
		\|\nabla u^-(s)\|_{L^2}^2\,ds
		\\
		&=
		2\chi_1
		\mathbb E
		\int_0^t
		\Theta_m^U(s)
		\int_{\mathcal O}
		u^-(s)\nabla u^-(s)\cdot\nabla w(s)\,dx\,ds
		\\
		&\quad
		-2
		\mathbb E
		\int_0^t
		\Theta_m^U(s)
		\int_{\mathcal O}
		u^-(s)F_1(u(s),v(s))\,dx\,ds
		\\
		&\quad
		+
		\mathbb E
		\int_0^t
		\int_{\{u(s)<0\}}
		\|\sigma_1(u(s))\|_{\ell^2}^2\,dx\,ds .
	\end{aligned}
	\]
	On the set $\{u<0\}$, one has $u=-u^-$ and $\nabla u=-\nabla u^-$, hence
	\[
	u\nabla u=u^-\nabla u^-.
	\]

	On the support of $\Theta_m^U$, the mild formula for $w$ and the $L^\infty$ bound on $u$ and $v$ give
	\[
	\|\nabla w(s)\|_{L^\infty}
	\le C_{m,T}.
	\]
	Therefore, Young's inequality gives
	\begin{equation}
		\label{eq:nonneg-chemo}
	\begin{aligned}
		&\quad2\chi_1
		\Theta_m^U(s)
		\left|
		\int_{\mathcal O}
		u^-\nabla u^-\cdot\nabla w\,dx
		\right|
		\\
		&\le
		\|\nabla u^-\|_{L^2}^2
		+
		C\Theta_m^U(s)
		\|\nabla w(s)\|_{L^\infty}^2
		\|u^-(s)\|_{L^2}^2
		\\
		&\le
		\|\nabla u^-\|_{L^2}^2
		+
		C_m
		\|u^-(s)\|_{L^2}^2 .
	\end{aligned}
	\end{equation}

	Next, on the set $\{u<0\}$, we have $u_+=0$. Therefore,
	\[
	F_1(u,v)
	=
	\mu_1u\left(1-a_1v_+^{\eta_1}\right).
	\]
	Since $u=-u^-$ on $\{u<0\}$, we obtain
	\[
	-2u^-F_1(u,v)
	=
	2\mu_1(u^-)^2
	\left(1-a_1v_+^{\eta_1}\right)
	\le
	2\mu_1(u^-)^2.
	\]
	Hence
	\begin{equation}
	\label{eq:nonneg-compete}
	-2
	\Theta_m^U(s)
	\int_{\mathcal O}
	u^-F_1(u,v)\,dx
	\le
	C\|u^-(s)\|_{L^2}^2.
	\end{equation}	
	Since $\sigma_1(0)=0$ and $\sigma_1$ is Lipschitz as an $\ell^2$-valued mapping, we have
	\[
	\|\sigma_1(r)\|_{\ell^2}
	\le
	K|r|,
	\qquad r\in\mathbb R.
	\]
	Therefore,
	\begin{equation} \label{eq:nonneg-stochastic}
	\int_{\{u<0\}}
	\|\sigma_1(u)\|_{\ell^2}^2\,dx
	\le
	C\|u^-\|_{L^2}^2.		
	\end{equation}

	Combining the above estimates \eqref{eq:nonneg-chemo}--\eqref{eq:nonneg-stochastic} and absorbing the gradient term, we get
	\[
	\mathbb E\|u^-(t)\|_{L^2}^2
	\le
		C_{m,T}
	\mathbb E\int_0^t
	\|u^-(s)\|_{L^2}^2\,ds.
	\]
	Gronwall's lemma implies
	\[
	\mathbb E\|u^-(t)\|_{L^2}^2=0,
	\qquad t\in[0,T].
	\]
	Thus,
	\[
	u^-(t)=0
	\quad\text{a.e. in }\mathcal O,\quad \mathbb P\text{-a.s.}
	\]
	for every fixed $t\in[0,T]$. Since $u$ has continuous trajectories in $C(\overline{\mathcal O})$, a standard argument based on a countable dense subset of $[0,T]$ yields
	\[
	u(t,x)\ge0,
	\qquad
	(t,x)\in[0,T]\times\mathcal{O},
	\quad \mathbb P\text{-a.s.}
	\]

	The proof for $v$ is identical. Applying It\^o's formula to $\phi_\delta(v(t))$ and passing to the limit $\delta\to0$, we obtain
	\[
	\begin{aligned}
		&\quad\mathbb E\|v^-(t)\|_{L^2}^2
		+
		2\mathbb E
		\int_0^t
		\|\nabla v^-(s)\|_{L^2}^2\,ds
		\\
		&=
		2\chi_2
		\mathbb E
		\int_0^t
		\Theta_m^U(s)
		\int_{\mathcal O}
		v^-(s)\nabla v^-(s)\cdot\nabla w(s)\,dx\,ds
		\\
		&\quad
		-2
		\mathbb E
		\int_0^t
		\Theta_m^U(s)
		\int_{\mathcal O}
		v^-(s)F_2(v(s),u(s))\,dx\,ds
		\\
		&\quad
		+
		\mathbb E
		\int_0^t
		\int_{\{v(s)<0\}}
		\|\sigma_2(v(s))\|_{\ell^2}^2\,dx\,ds .
	\end{aligned}
	\]
	On the set $\{v<0\}$, we have $v_+=0$, and hence
	\[
	F_2(v,u)
	=
	\mu_2v\left(1-a_2u_+^{\eta_2}\right).
	\]
	Therefore,
	\[
	-2v^-F_2(v,u)
	=
	2\mu_2(v^-)^2
	\left(1-a_2u_+^{\eta_2}\right)
	\le
	2\mu_2(v^-)^2.
	\]
	The chemotaxis term and the stochastic correction term are estimated in the same way as for $u$. Hence
	\[
	\mathbb E\|v^-(t)\|_{L^2}^2
	\le
		C_{m,T}\mathbb E
	\int_0^t	\|v^-(s)\|_{L^2}^2\,ds.
	\]
	By Gronwall's lemma,
	\[
	\mathbb E\|v^-(t)\|_{L^2}^2=0,
	\qquad t\in[0,T].
	\]
	Using the continuity of $v$ in $C(\overline{\mathcal O})$, we obtain
	\[
	v(t,x)\ge0,
	\qquad
	(t,x)\in[0,T]\times\mathcal{O},
	\quad \mathbb P\text{-a.s.}
	\]
	Therefore,
	\[
	u(t,x)\ge0,
	\qquad
	v(t,x)\ge0,
	\qquad
	(t,x)\in[0,T]\times\mathcal{O},
	\quad \mathbb P\text{-a.s.}
	\]
	The Neumann heat semigroup preserves nonnegativity. The mild formula for $w$, together with $w_0,u,v\ge0$, therefore gives $w(t,x)\ge0$ on $[0,T]\times\mathcal O$, almost surely.

	The proof is complete.
\end{proof}

\subsection{Uniqueness of local solutions}

In this subsection, we prove the uniqueness of local mild solutions to the original system. This result will be used later to identify the solutions obtained from different cut-off levels.

\begin{theorem}
	\label{thm:local-uniqueness}
	Assume that the coefficients $\sigma_i$, $i=1,2$, satisfy the Lipschitz condition stated above. Assume also that $F_1$ and $F_2$ are locally Lipschitz continuous on bounded subsets of $\mathbb R^2$. Let
	\[
	(u_1,v_1,w_1,\tau_1)
	\quad\text{and}\quad
	(u_2,v_2,w_2,\tau_2)
	\]
	be two local mild solutions to system \eqref{eq:main} with the same initial data $(u_0,v_0,w_0)$. Set
	\[
	\tau:=\tau_1\wedge\tau_2.
	\]
	Then
	\[
	u_1(t)=u_2(t),\qquad
	v_1(t)=v_2(t),\qquad
	w_1(t)=w_2(t),
	\]
	for all $t\in[0,\tau)$, almost surely.
\end{theorem}

\begin{proof}
	Let
	\[
	\tau:=\tau_1\wedge\tau_2.
	\]
	For $R>0$, define
	\[
	\tau_R
	:=
	\inf\left\{
	t\ge0:
	\max_{i=1,2}
	\sup_{0\le s\le t}
	\left(
	\|u_i(s)\|_{L^\infty}
	+
	\|v_i(s)\|_{L^\infty}
	\right)
	\ge R
	\right\}
	\wedge\tau.
	\]
	Fix $T>0$ and choose an integer $m>R$. On
	$[0,T\wedge\tau_R]$,
	\[
	R_t(u_i,v_i)\le R<m,
	\qquad i=1,2,
	\]
	and hence
	\[
	\Theta_m^{(u_i,v_i)}(t)=1.
	\]
	Therefore, up to $T\wedge\tau_R$, both
	$(u_1,v_1,w_1)$ and $(u_2,v_2,w_2)$ solve the same
	$m$-cut-off system with the same initial data. The uniqueness
	part of Theorem~\ref{thm:cutoff-existence}, applied to the
	stopped processes, gives
	\[
	u_1=u_2,\qquad v_1=v_2,\qquad w_1=w_2
	\quad\text{on }[0,T\wedge\tau_R]
	\]
	almost surely.
	
	Since $T>0$ is arbitrary, the equality holds on
	$[0,\tau_R]$. Letting $R\to\infty$ and using
	$\tau_R\uparrow\tau$ proves the result.
\end{proof}

\subsection{Construction of the maximal local solution}

For each $m\in\mathbb N$, let $X_m=(u_m,v_m,w_m)$ be the global solution of the cut-off system~\eqref{eq:cutoff-system}, and write $U_m=(u_m,v_m)$. Define
\[
\tau_m
:=
\inf\{t\ge0:R_t(U_m)\ge m\}.
\]
Before $\tau_m$, the cut-off factor is equal to one, so $X_m$ solves the original system~\eqref{eq:main}. If $m_2>m_1$, Theorem~\ref{thm:local-uniqueness} shows that $X_{m_2}=X_{m_1}$ up to $\tau_{m_1}$.

Set
\[
\tau:=\lim_{m\to\infty}\tau_m
\]
and define $X(t):=X_m(t)$ whenever $t\le\tau_m$. Thus $X=(u,v,w)$ is an adapted nonnegative mild solution of~\eqref{eq:main} on $[0,\tau)$, and Theorem~\ref{thm:local-uniqueness} gives uniqueness. Hence $(X,\tau)$ is the unique maximal local mild solution.

\section{Global Existence}

In this section, we prove that the maximal local solution constructed in Section~3 is global. We assume throughout that both species are subject to nonlinear self-damping:

\begin{assumption}\label{ass:self-damping}
	\begin{equation}
		\zeta_1>2,\qquad \zeta_2>2,\qquad\eta_1\ge1,\qquad \eta_2\ge1.
	\end{equation}
\end{assumption}

By Section~3, system~\eqref{eq:main} admits a maximal local mild solution $(u,v,w,\tau)$. Recall the stopping times

\[
\tau_m
:=
\inf\left\{
t\ge0:
\sup_{0\le s\le t}
\left(
\|u(s)\|_{L^\infty}
+
\|v(s)\|_{L^\infty}
\right)
\ge m
\right\},
\]
and $\tau_m\uparrow\tau$ almost surely. Our aim is to prove that $\tau=\infty$ almost surely.

We first record the stopped $L^p$ It\^o formula for the present system. The rigorous justification follows from the same Yosida approximation argument as in \cite[Theorem~4.1]{CHEN2025113531}; hence we only state the formula and then carry out the estimates needed here.

\begin{lemma}[$L^p$ It\^o formula]
	\label{lem:Lp-Ito-global}
	Let $p\ge2$ and $T>0$. Then for every $m\in\mathbb N$ and every $t\in[0,T]$, the following identities hold almost surely:
	\[
	\begin{aligned}
		\|u(t\wedge\tau_m)\|_{L^p}^p
		&-
		\|u_0\|_{L^p}^p
		+
		p(p-1)
		\int_0^{t\wedge\tau_m}
		\int_{\mathcal O}
		u^{p-2}|\nabla u|^2\,dx\,ds
		\\
		={}&
		p(p-1)\chi_1
		\int_0^{t\wedge\tau_m}
		\int_{\mathcal O}
		u^{p-1}\nabla u\cdot\nabla w\,dx\,ds
		\\
		&\quad+
		p\mu_1
		\int_0^{t\wedge\tau_m}
		\int_{\mathcal O}
		u^p\,dx\,ds
		-
		p\mu_1
		\int_0^{t\wedge\tau_m}
		\int_{\mathcal O}
		u^{p+\zeta_1}\,dx\,ds
		\\
		&\quad-
		p\mu_1a_1
		\int_0^{t\wedge\tau_m}
		\int_{\mathcal O}
		u^p v^{\eta_1}\,dx\,ds
		\\
		&\quad+
		p\sum_{k=1}^{\infty}
		\int_0^{t\wedge\tau_m}
		\int_{\mathcal O}
		u^{p-1}\sigma^1_{k}(u)\,dx\,dW_k^1(s)
		\\
		&\quad+
		\frac{p(p-1)}{2}
		\int_0^{t\wedge\tau_m}
		\int_{\mathcal O}
		u^{p-2}\|\sigma_1(u)\|_{\ell^2}^2\,dx\,ds ,
	\end{aligned}
	\]
	and
	\[
	\begin{aligned}
		\|v(t\wedge\tau_m)\|_{L^p}^p
		&-
		\|v_0\|_{L^p}^p
		+
		p(p-1)
		\int_0^{t\wedge\tau_m}
		\int_{\mathcal O}
		v^{p-2}|\nabla v|^2\,dx\,ds
		\\
		=&
		p(p-1)\chi_2
		\int_0^{t\wedge\tau_m}
		\int_{\mathcal O}
		v^{p-1}\nabla v\cdot\nabla w\,dx\,ds
		\\
		&\quad+
		p\mu_2
		\int_0^{t\wedge\tau_m}
		\int_{\mathcal O}
		v^p\,dx\,ds
		-
		p\mu_2
		\int_0^{t\wedge\tau_m}
		\int_{\mathcal O}
		v^{p+\zeta_2}\,dx\,ds
		\\
		&\quad-
		p\mu_2a_2
		\int_0^{t\wedge\tau_m}
		\int_{\mathcal O}
		v^p u^{\eta_2}\,dx\,ds
		\\
		&\quad+
		p\sum_{k=1}^{\infty}
		\int_0^{t\wedge\tau_m}
		\int_{\mathcal O}
		v^{p-1}\sigma^2_{k}(v)\,dx\,dW_k^2(s)
		\\
		&\quad+
		\frac{p(p-1)}{2}
		\int_0^{t\wedge\tau_m}
		\int_{\mathcal O}
		v^{p-2}\|\sigma_2(v)\|_{\ell^2}^2\,dx\,ds .
	\end{aligned}
	\]
\end{lemma}

The step toward proving global existence is the following a priori bound. It shows that, for every $p>2$, the two population components are uniformly bounded in $L^p(\mathcal O)$ with respect to the stopping time $\tau_m$.

\begin{lemma}
	\label{lem:Lp-bound-global}
	
	Under Assumption~\ref{ass:self-damping}, for every $p>2$ and $T>0$, there exists a constant $C_{p,T}>0$, independent of $m$, such that
	\[
	\mathbb E
	\sup_{0\le t\le T\wedge\tau_m}
	\left(
	\|u(t)\|_{L^p}^p
	+
	\|v(t)\|_{L^p}^p
	\right)
	\le
	C_{p,T}.
	\]
	Moreover,
	\[
	\mathbb E
	\int_0^{T\wedge\tau_m}
	\int_{\mathcal O}
	\left(
	u^{p+\zeta_1}
	+
	v^{p+\zeta_2}
	\right)\,dx\,dt
	\le
	C_{p,T}.
	\]

\end{lemma}

\begin{proof}

	For any $m\ge1$ and $t\in[0,T]$, set 
	\[
	Y_p(t):=\|u(t)\|_{L^p}^p+\|v(t)\|_{L^p}^p.
	\]
	Applying Lemma~\ref{lem:Lp-Ito-global} to $u$ and $v$, respectively, and adding the resulting identities, we obtain
	\begin{equation}\label{eq:complete-Lp-energy-identity}
	\begin{aligned}
	Y_p(t\wedge\tau_m)
	&+p(p-1)\int_0^{t\wedge\tau_m}\!\int_{\mathcal O}
	\left(u^{p-2}|\nabla u|^2+v^{p-2}|\nabla v|^2\right)\,dx\,ds\\
	&+p\mu_1\int_0^{t\wedge\tau_m}\!\int_{\mathcal O}u^{p+\zeta_1}\,dx\,ds
	+p\mu_2\int_0^{t\wedge\tau_m}\!\int_{\mathcal O}v^{p+\zeta_2}\,dx\,ds\\
	&+p\mu_1a_1\int_0^{t\wedge\tau_m}\!\int_{\mathcal O}u^p v^{\eta_1}\,dx\,ds
	+p\mu_2a_2\int_0^{t\wedge\tau_m}\!\int_{\mathcal O}v^p u^{\eta_2}\,dx\,ds\\
 	&=Y_p(0)\\
	&+p\mu_1\int_0^{t\wedge\tau_m}\!\int_{\mathcal O}u^p\,dx\,ds + p\mu_2\int_0^{t\wedge\tau_m}\!\int_{\mathcal O}v^p\,dx\,ds\\
	&+p(p-1)\chi_1\int_0^{t\wedge\tau_m}\!\int_{\mathcal O}
	u^{p-1}\nabla u\cdot\nabla w\,dx\,ds\\
	&+p(p-1)\chi_2\int_0^{t\wedge\tau_m}\!\int_{\mathcal O}
	v^{p-1}\nabla v\cdot\nabla w\,dx\,ds+M_p(t)\\
	&+\frac{p(p-1)}2\int_0^{t\wedge\tau_m}\!\int_{\mathcal O}
	\left(u^{p-2}\|\sigma_1(u)\|_{\ell^2}^2
	+v^{p-2}\|\sigma_2(v)\|_{\ell^2}^2\right)\,dx\,ds,
	\end{aligned}
	\end{equation}
	where
	\[
	\begin{aligned}
	M_p(t)&:=p\sum_{k=1}^{\infty}\int_0^{t\wedge\tau_m}\int_{\mathcal O}
	u^{p-1}\sigma_k^1(u)\,dx\,dW_k^1(s)\\
	&\quad+p\sum_{k=1}^{\infty}\int_0^{t\wedge\tau_m}\int_{\mathcal O}
	v^{p-1}\sigma_k^2(v)\,dx\,dW_k^2(s).
	\end{aligned}
	\]
	Since $u$ and $v$ are nonnegative, the two competition terms on the left-hand side of \eqref{eq:complete-Lp-energy-identity} are nonnegative. We therefore omit them when deriving the upper estimate.

	We first estimate the chemotactic terms. By Young's inequality,
	\begin{align}
	& p(p-1)\chi_1\int_0^{t\wedge\tau_m}\!\int_{\mathcal O}
	u^{p-1}\nabla u\cdot\nabla w\,dx\,ds
	+p(p-1)\chi_2\int_0^{t\wedge\tau_m}\!\int_{\mathcal O}
	v^{p-1}\nabla v\cdot\nabla w\,dx\,ds\notag\\
	&\quad\le \frac{p(p-1)}4\int_0^{t\wedge\tau_m}\!\int_{\mathcal O}
	\left(u^{p-2}|\nabla u|^2+v^{p-2}|\nabla v|^2\right)\,dx\,ds
	+C_p\bigl(I_u(t)+I_v(t)\bigr),
	\label{eq:chemotaxis-first-bound}
	\end{align}
	where
	\[
	I_u(t):=\int_0^{t\wedge\tau_m}\!\int_{\mathcal O}u^p|\nabla w|^2\,dx\,ds,
	\qquad
	I_v(t):=\int_0^{t\wedge\tau_m}\!\int_{\mathcal O}v^p|\nabla w|^2\,dx\,ds.
	\]
	To bound $I_u(t)$ and $I_v(t)$, we use the equation for $w$. Let $r:=p+2$,
	\[
	w(t)=e^{-t(A+1)}w_0+
	\int_0^t e^{-(t-s)(A+1)}\bigl(\alpha u(s)+\beta v(s)\bigr)\,ds
	\]
	Using Lemma~\ref{lem:heat-semigroup} and Young's convolution inequality in time, we have
	\begin{equation}\label{eq:w-gradient-controlled-by-uv}
	\|\nabla w\|_{L^r(0,t\wedge\tau_m;L^r)}^r
	\le C_{p,T}\left(
	1+\|w_0\|_{W^{1,\infty}}^r
	{}+\int_0^{t\wedge\tau_m}\!\int_{\mathcal O}(u^r+v^r)\,dx\,ds
	\right).
	\end{equation}
	By H\"older's inequality and Young's inequality,
	\[
	I_u(t)+I_v(t)
	\le C\int_0^{t\wedge\tau_m}\left(
	\|u\|_{L^r}^p+\|v\|_{L^r}^p\right)\|\nabla w\|_{L^r}^2\,ds
	\le C_p\int_0^{t\wedge\tau_m}\!\int_{\mathcal O}
	\left(u^r+v^r+|\nabla w|^r\right)\,dx\,ds.
	\]
	Since $r=p+2<p+\zeta_i$, $i=1,2$, substituting \eqref{eq:w-gradient-controlled-by-uv} into the above inequality and applying Young's inequality, we obtain, for every $\varepsilon>0$,
	\begin{equation}\label{eq:chemotaxis-absorbed}
	C_p\bigl(I_u(t)+I_v(t)\bigr)
	\le \varepsilon\int_0^{t\wedge\tau_m}\!\int_{\mathcal O}
	\left(u^{p+\zeta_1}+v^{p+\zeta_2}\right)\,dx\,ds
	+C_{\varepsilon,p,T}\left(1+\|w_0\|_{W^{1,\infty}}^{p+2}\right).
	\end{equation}

	The two lower-order reaction terms are bounded by
	\[
	C_p\int_0^{t\wedge\tau_m}Y_p(s)\,ds.
	\]
	Since $\sigma^i(0)=0$ and $\sigma^i$ is globally Lipschitz, $\|\sigma^i(z)\|_{\ell^2}\le K|z|$. Consequently, the It\^o correction admits the same bound. The Burkholder--Davis--Gundy inequality together with Young's inequality yields
	\begin{equation}\label{eq:Lp-BDG-bound}
	\mathbb E\sup_{0\le s\le t}|M_p(s)|
	\le \frac14\mathbb E\sup_{0\le s\le t\wedge\tau_m}Y_p(s)
	+C_p\mathbb E\int_0^{t\wedge\tau_m}Y_p(s)\,ds.
	\end{equation}
	Substituting \eqref{eq:chemotaxis-first-bound}, \eqref{eq:chemotaxis-absorbed}, and \eqref{eq:Lp-BDG-bound} into \eqref{eq:complete-Lp-energy-identity}, choosing $\varepsilon$ sufficiently small, and then taking the supremum over time and expectation, we deduce
	\begin{align*}
	&\mathbb E\sup_{0\le s\le t\wedge\tau_m}Y_p(s)
	+c_p\mathbb E\int_0^{t\wedge\tau_m}\!\int_{\mathcal O}
	\left(u^{p-2}|\nabla u|^2+v^{p-2}|\nabla v|^2\right)\,dx\,ds\\
	&\quad+c_p\mathbb E\int_0^{t\wedge\tau_m}\!\int_{\mathcal O}
	\left(u^{p+\zeta_1}+v^{p+\zeta_2}\right)\,dx\,ds\\
	&\le C_{p,T}\left(1+\mathbb EY_p(0)
	+\mathbb E\|w_0\|_{W^{1,\infty}}^{p+2}\right)
	+C_p\int_0^t\mathbb E\sup_{0\le r\le s\wedge\tau_m}Y_p(r)\,ds.
	\end{align*}
	The Gronwall inequality then implies
	\[
	\mathbb E\sup_{0\le s\le T\wedge\tau_m}Y_p(s)\le C_{p,T}.
	\]
	Substituting this estimate into the preceding inequality gives
	\[
	\mathbb E\int_0^{T\wedge\tau_m}\!\int_{\mathcal O}
	\left(u^{p+\zeta_1}+v^{p+\zeta_2}\right)\,dx\,ds
	\le C_{p,T}.
	\]
	As the constant is independent of $m$, the proof is complete.

\end{proof}

We next upgrade the uniform $L^p$ estimates to an $L^\infty$ estimate.

\begin{lemma}
	\label{lem:Linf-bound-global}
	For every $T>0$, there exists a constant $C_T>0$, independent of $m$, such that
	\[
	\mathbb E
	\sup_{0\le t\le T\wedge\tau_m}
	\left(
	\|u(t)\|_{L^\infty}
	+
	\|v(t)\|_{L^\infty}
	\right)
	\le
	C_T.
	\]
\end{lemma}

\begin{proof}
	Choose $p_0>2$ large enough that there is $\gamma>0$ satisfying
	\[
	\frac{n}{2p_0}<\gamma<\frac12.
	\]
	Taking $\varepsilon>0$ sufficiently small, we have
	\[
	\gamma+\frac12+\varepsilon<1.
	\]

	We estimate $u$ by its mild formulation. For $0\le t\le T\wedge\tau_m$, Lemma~\ref{lem:heat-semigroup}, the embedding
	\[
	D((A_{p_0}+1)^\gamma)\hookrightarrow C(\overline{\mathcal O}),
	\]
	and the choice of $\gamma$, we obtain
	\begin{equation}
		\label{eq:global-u-infty}
	\begin{aligned}
		\|u(t)\|_{L^\infty}
		&\le C\|u_0\|_{L^\infty}
		+C\int_0^t
		\|(A+1)^\gamma e^{-(t-s)A}\nabla\cdot(u\nabla w)\|_{L^{p_0}}\,ds \\
		&\quad+C\int_0^t
		\|(A+1)^\gamma e^{-(t-s)A}F_1(u,v)\|_{L^{p_0}}\,ds\\
		&\quad+
		\left\|
		\int_0^t
		e^{-(t-s)A}\sigma_1(u(s))\,dW_1(s)
		\right\|_{L^\infty}
		\\
		&\le
			C\|u_0\|_{L^\infty} +
		\int_0^t
		(t-s)^{-\gamma-\frac12-\varepsilon}
		\|u(s)\nabla w(s)\|_{L^{p_0}}\,ds
		\\
		&\quad+
		C
		\int_0^t
		(t-s)^{-\gamma}
		\|F_1(u(s),v(s))\|_{L^{p_0}}\,ds
		\\
		&\quad+
		\left\|
		\int_0^t
		e^{-(t-s)A}\sigma_1(u(s))\,dW_1(s)
		\right\|_{L^\infty}.
	\end{aligned}
	\end{equation}

	Taking the supremum over $[0,T\wedge\tau_m]$ and then expectation in \eqref{eq:global-u-infty}, we first estimate the chemotactic term. The mild formula for $w$ gives, for every finite $r$,
	\[
	\sup_{0\le s\le T\wedge\tau_m}\|\nabla w(s)\|_{L^r}
	\le C_T\left(
	\|w_0\|_{W^{1,\infty}}
	+\sup_{0\le s\le T\wedge\tau_m}
	\bigl(\|u(s)\|_{L^r}+\|v(s)\|_{L^r}\bigr)
	\right).
	\]
	Thus Lemma~\ref{lem:Lp-bound-global}, used with sufficiently large exponents, and H\"older's inequality give

	\[
	\begin{aligned}
		&\quad\mathbb E
		\sup_{0\le t\le T\wedge\tau_m}
		\int_0^t
		(t-s)^{-\gamma-\frac12-\varepsilon}
		\|u(s)\nabla w(s)\|_{L^{p_0}}\,ds
		\\
		&\le
		C
		\mathbb E
		\left[
		\sup_{0\le s\le T\wedge\tau_m}
		\|u(s)\|_{L^{2p_0}}
		\sup_{0\le s\le T\wedge\tau_m}
		\|\nabla w(s)\|_{L^{2p_0}}
		\right]
		\sup_{0\le t\le T}\int_0^t
		(t-s)^{-\gamma-\frac12-\varepsilon}\,ds
		\\
		&\le
		C_T.
	\end{aligned}
	\]

	Next we estimate the reaction term. Since $F_1$ has polynomial growth,
	\[
	|F_1(u,v)|
	\le
	C
	\left(
	|u|+|u|^{1+\zeta_1}+|u||v|^{\eta_1}
	\right).
	\]
	Therefore,
	\[
	\begin{aligned}
		\|F_1(u(s),v(s))\|_{L^{p_0}}
		&\le
		C
		\left(
		\|u(s)\|_{L^{p_0}}
		+
		\|u(s)\|_{L^{(1+\zeta_1)p_0}}^{1+\zeta_1}
		\right.
		\\
		&\qquad\left.
		+
		\|u(s)v(s)^{\eta_1}\|_{L^{p_0}}
		\right).
	\end{aligned}
	\]
	For the mixed term, by H\"older's inequality,
	\[
	\|u(s)v(s)^{\eta_1}\|_{L^{p_0}}
	\le
	\|u(s)\|_{L^{2p_0}}
	\|v(s)\|_{L^{2\eta_1p_0}}^{\eta_1}.
	\]
	Applying Lemma~\ref{lem:Lp-bound-global} with the exponents appearing above, we get
	\[
	\mathbb E
	\sup_{0\le t\le T\wedge\tau_m}
	\int_0^t
	(t-s)^{-\gamma}
	\|F_1(u(s),v(s))\|_{L^{p_0}}\,ds
	\le
	C_T,
	\]
	because $\gamma<1$.

	For the stochastic convolution, choose $q_0>3$ and use Lemma~\ref{lem:stochastic-convolution} with $\eta=1$, $r=\frac32$ and $q=q_0$. Then
	\[
	\begin{aligned}
		&\mathbb E
		\sup_{0\le t\le T\wedge\tau_m}
		\left\|
		\int_0^t
		e^{-(t-s)A}\sigma_1(u(s))\,dW_1(s)
		\right\|_{L^\infty}
		\\
		&\qquad\le
		C(1+T)^{C'}
		\mathbb E
		\left(
		\int_0^{T\wedge\tau_m}
		\|\|\sigma_1(u(s))\|_{\ell^2}\|_{L^{2q_0}}^3\,ds
		\right)^{1/3}.
	\end{aligned}
	\]
	By the linear growth of $\sigma_1$,
	\[
	\|\|\sigma_1(u)\|_{\ell^2}\|_{L^{2q_0}}
	\le
	C
	\left(
	1+\|u\|_{L^{2q_0}}
	\right).
	\]
	Thus, we have
	\[
	\mathbb E
	\sup_{0\le t\le T\wedge\tau_m}
	\left\|
	\int_0^t
	e^{-(t-s)A}\sigma_1(u(s))\,dW_1(s)
	\right\|_{L^\infty}
	\le
	C_T.
	\]

	Combining the estimates for the chemotactic term, the reaction term and the stochastic convolution, we obtain
	\[
	\mathbb E
	\sup_{0\le t\le T\wedge\tau_m}
	\|u(t)\|_{L^\infty}
	\le
	C_T.
	\]

	The proof for $v$ is the same. Indeed, using the mild formulation of $v$, the estimate
	\[
	|F_2(v,u)|
	\le
	C
	\left(
	|v|+|v|^{1+\zeta_2}+|v||u|^{\eta_2}
	\right),
	\]
	and the stochastic convolution estimate for $\sigma_2(v)$, we obtain
	\[
	\mathbb E
	\sup_{0\le t\le T\wedge\tau_m}
	\|v(t)\|_{L^\infty}
	\le
	C_T.
	\]
	Therefore,
	\[
	\mathbb E
	\sup_{0\le t\le T\wedge\tau_m}
	\left(
	\|u(t)\|_{L^\infty}
	+
	\|v(t)\|_{L^\infty}
	\right)
	\le
	C_T.
	\]
	The proof is complete.
\end{proof}
The preceding uniform estimates yield global well-posedness.
\begin{theorem}
	\label{thm:main-result}
	Suppose that Assumptions~\ref{ass:stochastic-sigma}, \ref{ass:initial-data}, and~\ref{ass:self-damping} hold, and that all deterministic coefficients in system~\eqref{eq:main} are positive. Then the system admits a unique global adapted nonnegative mild solution $(u,v,w)$. For every $T>0$ and every $q>n$,
	\[
	u,v\in L^1\bigl(\Omega;C([0,T];C(\overline{\mathcal O}))\bigr)
	\]
	and
	\[
	w\in L^1\bigl(\Omega;C([0,T];W^{1,q}(\mathcal O))\bigr).
	\]
	Moreover,
	\[
	u(t,x)\ge0,\qquad v(t,x)\ge0,\qquad w(t,x)\ge0
	\]
	for all $t\in[0,T]$ and $x\in\overline{\mathcal O}$, almost surely.
\end{theorem}

\begin{proof}
	Let $(u,v,w,\tau)$ be the maximal local mild solution. By Lemma~\ref{lem:Linf-bound-global}, for every $T>0$,
	\[
	\mathbb E
	\sup_{0\le t\le T\wedge\tau_m}
	\left(
	\|u(t)\|_{L^\infty} + \|v(t)\|_{L^\infty}
	\right)
	\le
	C_T,
	\]
	where $C_T$ is independent of $m$.

	On the event $\{\tau_m\le T\}$, by the definition of $\tau_m$,
	\[
	\sup_{0\le t\le T\wedge\tau_m}
	\left(
	\|u(t)\|_{L^\infty} + \|v(t)\|_{L^\infty}
	\right)
	\ge m.
	\]
	Therefore, by Markov's inequality,
	\[
	\mathbb P(\tau_m\le T)	\le	\frac{1}{m}	\mathbb E\sup_{0\le t\le T\wedge\tau_m}
	\left( \|u(t)\|_{L^\infty} + \|v(t)\|_{L^\infty} \right) \le \frac{C_T}{m}.
	\]
	Letting $m\to\infty$, and using $\tau_m\uparrow\tau$, we obtain
	\[
	\mathbb P(\tau\le T)=0.
	\]
	Hence $\tau\ge T$ almost surely for every $T>0$, proving global existence. Uniqueness follows from Theorem~\ref{thm:local-uniqueness}, and nonnegativity follows from Theorem~\ref{thm:cutoff-existence} and the localization construction. Lemma~\ref{lem:Linf-bound-global} gives the stated integrability of $u$ and $v$; the corresponding assertion for $w$ follows from its mild formula and the heat-semigroup estimates.
\end{proof}

\section{Continuous Dependence on Finite Time Intervals}
\label{sec:finite-time-continuous-dependence}

Fix $q>n$ and set
\[
\mathcal X_q
:=
C(\overline{\mathcal O})
\times C(\overline{\mathcal O})
\times W^{1,q}(\mathcal O).
\]
We equip this space with the norm
\[
\|(u,v,w)\|_{\mathcal X_q}
:=
\|u\|_{L^\infty}
+\|v\|_{L^\infty}
+\|w\|_{W^{1,q}}.
\]
We prove continuous dependence by using the estimates from Theorem~\ref{thm:local-uniqueness}. A related argument for stochastic two-species Lotka--Volterra reaction--diffusion systems appears in~\cite[Proposition~3.2]{NguyenYin2021}.

\begin{theorem}
	\label{thm:initial-data-continuity}
	Fix $T>0$, $p\ge2$, and $q>n$. Let \(X_{0,k}\) and \(X_0\) satisfy Assumption~\ref{ass:initial-data}, and let  $X_k=(u_k,v_k,w_k)$, $k\in\mathbb N$, and $X=(u,v,w)$ be the global mild solutions corresponding to initial data $X_{0,k}$ and $X_0$, respectively. Assume that all systems have the same coefficients, are driven by the same Wiener processes, and
	\[
	X_{0,k}\longrightarrow X_0
	\quad\text{in }L^p(\Omega;\mathcal X_q).
	\]
	Suppose further that
	\[
	\sup_{k\ge1}
	\mathbb E\sup_{0\le t\le T}\|X_k(t)\|_{\mathcal X_q}
	+
	\mathbb E\sup_{0\le t\le T}\|X(t)\|_{\mathcal X_q}
	<\infty.
	\]
	Then
	\[
	X_k\longrightarrow X
	\quad\text{in probability in }C([0,T];\mathcal X_q).
	\]
\end{theorem}

\begin{proof}
	Set
	\[
	\overline X_k:=X_k-X,
	\qquad
	\overline X_{0,k}:=X_{0,k}-X_0,
	\]
	and use the analogous notation for the components. By assumption,
	\begin{equation}
		\label{eq:uniform-finite-time-Xq-bound}
		\sup_{k\ge1}
		\mathbb E\sup_{0\le t\le T}\|X_k(t)\|_{\mathcal X_q}
		+
		\mathbb E\sup_{0\le t\le T}\|X(t)\|_{\mathcal X_q}
		<\infty.
	\end{equation}
	For $R>0$, define
	\[
	\rho_{k,R}
	:=
	\inf\left\{
	t\in[0,T]:
	\|X_k(t)\|_{\mathcal X_q}
	\vee
	\|X(t)\|_{\mathcal X_q}
	\ge R
	\right\}\wedge T.
	\]
	On $[0,\rho_{k,R}]$, the chemotaxis products satisfy
	\begin{equation}
		\label{eq:stability-chemotaxis-products}
		\begin{aligned}
			&\|u_k\nabla w_k-u\nabla w\|_{L^q}
			+\|v_k\nabla w_k-v\nabla w\|_{L^q}
			\\
			&\qquad\le
			C_R\|\overline X_k\|_{\mathcal X_q}.
		\end{aligned}
	\end{equation}
	Since $F_1$ and $F_2$ are locally Lipschitz, on $[0,\rho_{k,R}]$ we also have
	\begin{equation}
		\label{eq:stability-reaction-products}
		\begin{aligned}
			&\quad\|F_1(u_k,v_k)-F_1(u,v)\|_{L^\infty}+
			\|F_2(v_k,u_k)-F_2(v,u)\|_{L^\infty}\le
			C_R\|\overline X_k\|_{\mathcal X_q}.
		\end{aligned}
	\end{equation}
	
	Choose $\gamma\in(n/(2q),1/2)$ and $\varepsilon>0$ such that
	\[
	\gamma+\frac12+\varepsilon<1,
	\]
	and fix $\gamma'\in(1/2,1)$. Subtracting the mild formulations and arguing as in Theorem~\ref{thm:local-uniqueness}, with \eqref{eq:stability-chemotaxis-products}--\eqref{eq:stability-reaction-products} and Lemma~\ref{lem:stochastic-convolution}, gives, for $0<h\le T$,
	\begin{equation}
		\label{eq:short-time-stability}
		\begin{aligned}
			D_{k,R}(h)
			&:=
			\mathbb E\sup_{0\le t\le h\wedge\rho_{k,R}}
			\|\overline X_k(t)\|_{\mathcal X_q}^p
			\\
			&\le
			C\mathbb E\|\overline X_{0,k}\|_{\mathcal X_q}^p
			+\kappa_{p,R}(h)D_{k,R}(h),
		\end{aligned}
	\end{equation}
	where
	\[
	\kappa_{p,R}(h)
	:=
	C_{p,R}
	\left(
	h^{p(1-(\gamma+\frac12+\varepsilon))}
	+h^{p(1-\gamma')}
	+h^p
	+h^{p/3}
	\right)
	\longrightarrow0
	\qquad\text{as }h\downarrow0.
	\]
	Here the last term in $\kappa_{p,R}$ comes from the global Lipschitz continuity of $\sigma_i$ and the stochastic-convolution estimate.
	
	Choose $h_0>0$, independently of $k$, such that $\kappa_{p,R}(h_0)\le1/2$. After absorbing the second term in~\eqref{eq:short-time-stability}, apply the same estimate on the successive intervals of length $h_0$, using the difference at the left endpoint as the new initial value. A finite induction gives
	\begin{equation}
		\label{eq:localized-finite-time-stability}
		\mathbb E\sup_{0\le t\le\rho_{k,R}}
		\|X_k(t)-X(t)\|_{\mathcal X_q}^p
		\le
		C_{p,T,R}
		\mathbb E\|X_{0,k}-X_0\|_{\mathcal X_q}^p.
	\end{equation}
	
	For every $\delta>0$,~\eqref{eq:localized-finite-time-stability} and Markov's inequality imply
	\begin{align*}
		&\mathbb P\left\{
		\sup_{0\le t\le T}\|X_k(t)-X(t)\|_{\mathcal X_q}>\delta
		\right\}
		\\
		&\quad\le
		\mathbb P\{\rho_{k,R}<T\}
		+\delta^{-p}C_{p,T,R}
		\mathbb E\|X_{0,k}-X_0\|_{\mathcal X_q}^p.
	\end{align*}
	Moreover,~\eqref{eq:uniform-finite-time-Xq-bound} gives
	\[
	\lim_{R\to\infty}
	\sup_{k\ge1}\mathbb P\{\rho_{k,R}<T\}
	=0.
	\]
	For fixed $R$, the second term tends to zero as $k\to\infty$. Letting first $k\to\infty$ and then $R\to\infty$ proves the claim.
\end{proof}

\section{Conclusion}

	Chemotactic aggregation, interspecific competition, and environmental fluctuations may all influence the spatial distribution of populations that share a chemical signal. Our model includes these effects in a single system. Both populations produce and follow a parabolic signal, interact through nonlinear competition, and are directly affected by multiplicative noise. This provides a mathematical setting for asking whether population damping can prevent a loss of well-posedness when aggregation and random perturbations occur together.

	We proved that superquadratic self-damping gives a unique global adapted nonnegative mild solution. We also established continuous dependence in probability on admissible initial data over every finite time interval, with the signal measured in $W^{1,q}(\mathcal O)$ for $q>n$ and under a finite-time moment bound. Global existence and nonnegativity ensure that the two population densities and the signal concentration remain well defined for all finite times covered by the model. Continuous dependence shows that small changes in their initial profiles produce, in probability, only small changes over a fixed finite interval. These conclusions concern well-posedness and stability; they do not determine whether the populations coexist or whether one population eventually excludes the other.

	Due to the two chemotactic fluxes and the stochastic terms, the $L^p$ It\^o identities contain coupled terms with no fixed sign, and all bounds used to extend the solution must be uniform in the cut-off level. Parabolic estimates for the signal equation, together with superquadratic self-damping, control the powers generated by chemotaxis. Stochastic convolution estimates then provide the $L^\infty$ bounds required for global well-posedness.

	The quadratic self-damping case, higher-dimensional domains, and noise acting on the signal equation are not covered by the present proof. Long-time behavior under random perturbations, including coexistence, competitive exclusion, and spatial patterns, remains open.


	\bibliographystyle{spbasic}
	\bibliography{refer}

\end{document}